\documentclass[11pt]{article}

\usepackage{a4wide}
\usepackage{amsmath,amsbsy,amssymb,latexsym,amsthm}
\usepackage{pstricks}
\usepackage{graphicx}

\newtheorem{thm}{Theorem}
\newtheorem{lem}[thm]{Lemma}
\newtheorem{example}{Example}
\newtheorem{defn}[thm]{Definition}
\newtheorem{rem}{Remark}

\renewenvironment{abstract}{\begin{center}
\begin{minipage}[c]{14cm} {\begin{center}\bf Abstract \end{center}}} {\end{minipage}
\end{center}}
\newenvironment{keywords}{\begin{center}
\begin{minipage}[c]{14cm} {\bf Keywords:}} {\end{minipage}
\end{center}}
\newenvironment{msc}{\begin{center}
\begin{minipage}[c]{14cm} {\bf Mathematics Subject Classification 2010:}} {\end{minipage}
\end{center}}

\begin{document}

\title{Isoperimetric problems of the calculus of variations\\
with fractional derivatives\footnote{Submitted 02-Oct-2009; 
revised 30-Jun-2010; accepted 10-May-2011; 
\emph{Acta Mathematica Scientia}.}}

\author{Ricardo Almeida$^1$\\
\texttt{ricardo.almeida@ua.pt}
\and Rui A. C. Ferreira$^2$\\
\texttt{ruiacferreira@ua.pt}
\and Delfim F. M. Torres$^1$\\
\texttt{delfim@ua.pt}}

\date{$^1$Department of Mathematics\\
University of Aveiro\\
3810-193 Aveiro, Portugal\\[0.3cm]
$^2$Department of Mathematics\\
Faculty of Engineering and Natural Sciences\\
Lusophone University of Humanities and Technologies\\
1749-024 Lisbon, Portugal}

\maketitle


\begin{abstract}
In this paper we study isoperimetric problems of the
calculus of variations with left and right
Riemann-Liouville fractional derivatives.
Both situations when the lower bound of the variational
integrals coincide and do not coincide with the lower bound
of the fractional derivatives are considered.
\end{abstract}

\begin{msc}
49K05, 26A33.
\end{msc}



\section{Introduction}

Isoperimetric problems consist in maximizing or minimizing
a cost functional subject to integral constraints \cite{MR2546784}.
They have found a broad class of
important applications throughout the centuries.
Areas of application include astronomy, physics, geometry,
algebra, and analysis \cite{MR2528200,Viktor}.
Concrete isoperimetric problems in engineering have
been also investigated by a number of authors \cite{Curtis}.

The study of isoperimetric problems
is nowadays done, in an elegant and rigorously way,
by means of the theory of the calculus of variations.
This is possible through a powerful tool known as
the Euler-Lagrange equation \cite{Brunt}.
Recently the theory of the calculus of variations
has been considered in the fractional context
\cite{RicDel,comRic:Leitmann,MR2424086,Bal,MR2425823,MR2411429,MR2401415,%
Frederico:Torres10,J:M:B,comBasia:Frac1,MR2246299}.
The fractional calculus allows to generalize
the ordinary differentiation and integration
to an arbitrary (non-integer) order, and provides
a powerful tool for modeling and solving various
problems in science and engineering
\cite{Miller,Podlubny,samko}. The problems
considered are more general, and hold for
a bigger class of admissible
functions which are not necessarily differentiable
in the classical sense \cite{Ross:Samko:Love}.
Several results were proved for the new calculus
of variations. They include: Euler-Lagrange equations
for fractional variational problems
with Riemann-Liouville \cite{AGRA1},
Riesz \cite{AGRA3}, Caputo, and $(\alpha,\beta)$
derivatives \cite{El-Nabulsi:Torres07};
transversality conditions \cite{AGRA2};
and Noether's symmetry theorem
\cite{Frederico:Torres07,Frederico:Torres08}.
For a state of the art of the fractional
variational theory see the recent papers
\cite{MyID:182,Atanackovic,MR2563910,MyID:179,El-Nabulsi:Torres08,MR2550821}
and references therein. In this paper
we develop further the theory of the
fractional variational calculus
by studying isoperimetric problems.

The paper is organized as follows. In Section~\ref{sec2}
we shortly review the necessary background on fractional calculus.
Our results are given in Section~\ref{sec:MR}. In Section~\ref{sec3}
we introduce the basic fractional isoperimetric problem and prove
correspondent necessary optimality conditions, both for normal
and abnormal extremizers (Theorems~\ref{IsoPro} and \ref{IsoPro2},
respectively). In Section~\ref{sec4} we generalize our results
for functionals where the lower bound of the integral is greater
than the lower bound of the Riemann-Liouville derivatives. Finally,
in Section~\ref{sec5} we present a necessary condition of optimality
for the case where the order of the derivative is taken as a free variable.


\section{Preliminaries of fractional calculus}
\label{sec2}

A fractional derivative is a generalization of the ordinary differentiation,
which allows real number powers of the differential operator.
There exist numerous applications of fractional derivatives
to several fields, like geometry, physics, engineering, etc.
In the literature we may find a great number of definitions
for fractional derivatives (see, \textrm{e.g.}, \cite{Miller,Podlubny,samko}).
In this paper we deal with the left and right Riemann-Liouville
fractional derivatives, which are defined in the following way.

\begin{defn}
Let $f:[a,b]\rightarrow\mathbb{R}$ be a continuous function.
The left and right Riemann-Liouville fractional
derivatives of order $\alpha>0$ are defined respectively by
$${_a\mathcal{D}_x^\alpha}
f(x)=\frac{1}{\Gamma(n-\alpha)}\frac{d^n}{dx^n}\int_a^x
(x-t)^{n-\alpha-1}f(t)dt \, , \quad x \in (a,b] \, ,$$
and
$${_x\mathcal{D}_b^\alpha}
f(x)=\frac{(-1)^n}{\Gamma(n-\alpha)}\frac{d^n}{dx^n}\int_x^b
(t-x)^{n-\alpha-1} f(t)dt \, , \quad x \in [a,b) \, ,$$
where $\Gamma$ is the Euler gamma function,
$\alpha$ is the order of the derivative, and $n = [\alpha] + 1$
with $[\alpha]$ being the integer part of $\alpha$.
\end{defn}

If $\alpha\geq 1$ is an integer, these fractional derivatives are
understood in the sense of usual differentiation, that is,
$$
{_a\mathcal{D}_x^\alpha} f(x)
=\left(\frac{d}{dx}\right)^\alpha f(x)
\quad \mbox{ and } \quad {_x\mathcal{D}_b^\alpha}
f(x)=\left(-\frac{d}{dx}\right)^\alpha f(x).
$$

>From the physical point of view, if $f(x)$ describes
a certain process through time $x$, then the left derivative is related to
the past of this process, while the right derivative belongs to the future.

These operations are linear, in the sense that
$${_a\mathcal{D}_x^\alpha} (\mu f(x)+\nu g(x))=
\mu\, {_a\mathcal{D}_x^\alpha} f(x)+\nu \, {_a\mathcal{D}_x^\alpha}
g(x)$$ and $${_x\mathcal{D}_b^\alpha} (\mu f(x)+\nu g(x))= \mu \,
{_x\mathcal{D}_b^\alpha} f(x)+\nu \, {_x\mathcal{D}_b^\alpha} g(x).$$

We now present the integration by parts formula for fractional
derivatives.

\begin{lem}{\rm (\cite[p.~46]{samko})}
If $f$ and $g$ and the fractional derivatives
${_a\mathcal{D}_x^\alpha} g$ and ${_x\mathcal{D}_b^\alpha} f$
are continuous at every point $x\in[a,b]$, then for $0<\alpha<1$
we have
\begin{equation}
\label{integracao:partes}
\int_{a}^{b}f(x) {_a\mathcal{D}_x^\alpha} g(x)dx
=\int_a^b g(x) {_x\mathcal{D}_b^\alpha} f(x)dx.
\end{equation}
Moreover, formula (\ref{integracao:partes}) is still
valid for $\alpha=1$ provided $f$ or $g$
are zero at $x=a$ and $x=b$.
\end{lem}


\section{Main results}
\label{sec:MR}

>From now on we fix $\alpha,\beta \in (0,1)$.
We consider functionals $\mathcal{J}$ of the form
\begin{equation}
\label{P0}
\mathcal{J}(y)=\int_a^b L(x,y,\,
{_a\mathcal{D}_x^\alpha} y,\, {_x\mathcal{D}_b^\beta} y) dx
\end{equation}
defined on the set of admissible functions $y$ that have continuous
left fractional derivatives of order $\alpha$ and continuous right
fractional derivatives of order $\beta$ in $[a,b]$, and where
$(x,y,u,v)\to L(x,y,u,v)$ is a function with continuous first
and second partial derivatives with respect to all its arguments
such that $\frac{\partial L}{\partial u}(x,y,\,
{_a\mathcal{D}_x^\alpha} y,\, {_x\mathcal{D}_b^\beta} y)$
has continuous right fractional derivative of order $\alpha$
for all $x\in[a,b]$ and $\frac{\partial L}{\partial v}(x,y,\,
{_a\mathcal{D}_x^\alpha} y,\, {_x\mathcal{D}_b^\beta} y)$
has continuous left fractional derivative of order $\beta$ in $[a,b]$.

\begin{rem}
\label{rem:bc}
The left Riemann-Liouville fractional derivative is infinite at
$x=a$ if $y(a) \not=0$. If $y(b)\not=0$, then the right
Riemann-Liouville fractional derivative is also not finite
at $x=b$ \cite{Ross:Samko:Love}. For this reason, by considering
that the admissible functions $y$ have continuous left fractional derivatives,
then necessarily $y(a)=0$; by considering that the admissible functions $y$
have continuous right fractional derivatives, then necessarily $y(b)=0$.
This fact seems to have been neglected in some previous work on the calculus
of variations with Riemann-Liouville fractional derivatives.
Alternatively, we can consider the general case of boundary conditions,
say $y(a)=y_a$ and $y(b)=y_b$, and study functionals of type
$$
\mathcal{J}(y)=\int_a^b L(x,y(x),\,{_a\mathcal{D}_x^\alpha}
(y(x)-y_a),\, {_x\mathcal{D}_b^\beta} (y(x)-y_b)) dx.
$$
This needs, however, a modified fractional calculus \cite{jumarie}.
\end{rem}

\begin{defn}
The functional $\mathcal{J}$ is said to have a \emph{local minimum}
(resp. \emph{local maximum}) at $y$ if there exists a $\delta>0$
such that $\mathcal{J}(y)\leq \mathcal{J}(y_1)$
(resp. $\mathcal{J}(y)\geq \mathcal{J}(y_1)$)
for all $y_1$ such that $\|y-y_1\|<\delta$.
\end{defn}

In \cite{AGRA1} the following problem is addressed: among all
curves $y(x)$ satisfying the boundary conditions, find the ones
that maximize or minimize a given functional $\mathcal{J}$.
An answer to this question is given in the next theorem.

\begin{thm}[\cite{AGRA1}]
\label{ELtheorem}
Let $\mathcal{J}$ be a functional as in (\ref{P0})
and  $y$ an extremum of $\mathcal{J}$. Then, $y$
satisfies the following Euler-Lagrange equation:
\begin{equation}
\label{ELequation}
\frac{\partial L}{\partial y}+{_x\mathcal{D}_b^\alpha}
\frac{\partial L}{\partial u}+{_a\mathcal{D}_x^\beta}
\frac{\partial L}{\partial v}=0.
\end{equation}
\end{thm}


\subsection{The fractional isoperimetric problem}
\label{sec3}

We introduce the \emph{fractional isoperimetric problem}
as follows: find the functions $y$ that satisfy
boundary conditions
\begin{equation}
\label{bouncond}
y(a)=y_a, \quad y(b)=y_b
\end{equation}
($y_a = 0$ if left Riemann-Liouville fractional derivatives
are present in (\ref{P0}); $y_b = 0$ if right Riemann-Liouville
fractional derivatives are present in (\ref{P0})
--- \textrm{cf.} Remark~\ref{rem:bc}), the integral constraint
\begin{equation}
\label{isocons}
\mathcal{I}(y)=\int_a^b g(x,y,\, {_a\mathcal{D}_x^\alpha} y,\,
{_x\mathcal{D}_b^\beta} y)dx=l \, ,
\end{equation}
and give a minimum or a maximum to (\ref{P0}).
We assume that $l$ is a specified real constant,
functions $y$ have continuous left and right
fractional derivatives (if present in (\ref{P0})),
and $(x,y,u,v)\to g(x,y,u,v)$ is a function with continuous
first and second partial derivatives with respect to all
its arguments such that $\frac{\partial g}{\partial u}(x,y,\,
{_a\mathcal{D}_x^\alpha} y,\, {_x\mathcal{D}_b^\beta} y)$
has continuous right fractional derivative of order $\alpha$
for all $x\in[a,b]$ and $\frac{\partial g}{\partial v}(x,y,\,
{_a\mathcal{D}_x^\alpha} y,\, {_x\mathcal{D}_b^\beta} y)$
has continuous left fractional derivative of order $\beta$
in $[a,b]$. Theorem~\ref{ELtheorem} motivates the following definition.

\begin{defn}
\label{extermal}
An admissible function $y$ is an \emph{extremal} for $\mathcal{I}$
in (\ref{isocons}) if it satisfies the equation
$$
\frac{\partial g}{\partial y}+{_x\mathcal{D}_b^\alpha}
\frac{\partial g}{\partial u}+{_a\mathcal{D}_x^\beta}\frac{\partial g}{\partial v}=0
$$
for all $x\in[a,b]$.
\end{defn}

The following theorem gives a necessary condition for
$y$ to be a solution of the fractional isoperimetric
problem defined by (\ref{P0})-(\ref{bouncond})-(\ref{isocons})
under the assumption that $y$ is not an extremal for $\mathcal{I}$.

\begin{thm}
\label{IsoPro}
Suppose that $\mathcal{J}$ given by (\ref{P0})
has a local minimum or a local maximum at $y$ subject
to the boundary conditions (\ref{bouncond}) and the isoperimetric
constraint (\ref{isocons}). Further, suppose that $y$ is not
an extremal for the functional $\mathcal{I}$. Then there exists
a constant $\lambda$ such that $y$ satisfies
the fractional differential equation
\begin{equation}
\label{E-L}
\frac{\partial F}{\partial y}
+{_x\mathcal{D}_b^\alpha}
\frac{\partial F}{\partial u}
+{_a\mathcal{D}_x^\beta}\frac{\partial F}{\partial v}=0
\end{equation}
with $F=L-\lambda g$.
\end{thm}

\begin{proof}
Consider neighboring functions of the form
\begin{equation}
\label{admfunct}
\hat{y}=y+\epsilon_1\eta_1+\epsilon_2\eta_2,
\end{equation}
where for each $i\in\{1,2\}$ $\epsilon_i$ is a sufficiently small parameter,
$\eta_i$ have continuous left and right fractional derivatives,
and $\eta_i(a)=\eta_i(b)=0$.

First we will show that (\ref{admfunct}) has a subset of admissible
functions for the fractional isoperimetric problem. Consider the quantity
$$
\mathcal{I}(\hat{y})=\int_a^b g(x,y+\epsilon_1\eta_1+\epsilon_2\eta_2,
{_a\mathcal{D}_x^\alpha} y+\epsilon_1  {_a\mathcal{D}_x^\alpha} \eta_1
+\epsilon_2 {_a\mathcal{D}_x^\alpha} \eta_2,
{_x\mathcal{D}_b^\beta} y+\epsilon_1 {_x\mathcal{D}_b^\beta} \eta_1
+\epsilon_2 {_x\mathcal{D}_b^\beta}\eta_2)dx.
$$
Then we can regard $\mathcal{I}(\hat{y})$ as a function
of $\epsilon_1$ and $\epsilon_2$. Define
$\hat{I}(\epsilon_1,\epsilon_2)=\mathcal{I}(\hat{y})-l$.
Thus,
\begin{equation}
\label{implicit1}
\hat I(0,0)=0.
\end{equation}
On the other hand, we have
\begin{align}
\left.\frac{\partial \hat I}{\partial \epsilon_2} \right|_{(0,0)}
&=\int_a^b\left[ \frac{\partial g}{\partial y}\eta_2
+\frac{\partial g}{\partial u} {_a\mathcal{D}_x^\alpha} \eta_2
+ \frac{\partial g}{\partial v} {_x\mathcal{D}_b^\beta} \eta_2 \right]dx\nonumber\\
&=\int_a^b \left[ \frac{\partial g}{\partial y}
+{_x\mathcal{D}_b^\alpha} \frac{\partial g}{\partial u}
+ {_a\mathcal{D}_x^\beta} \frac{\partial g}{\partial v}\right]\eta_2 dx\label{rui0},
\end{align}
where (\ref{rui0}) follows from (\ref{integracao:partes}).
Since $y$ is not an extremal for $\mathcal{I}$,
by the fundamental lemma of the calculus of
variations (see, \textrm{e.g.}, \cite[p.~32]{Brunt}),
there exists a function $\eta_2$ such that
\begin{equation}
\label{implicit2}
\left.\frac{\partial \hat I}{\partial \epsilon_2} \right|_{(0,0)}\neq 0.
\end{equation}
Using (\ref{implicit1}) and (\ref{implicit2}), the implicit
function theorem asserts that there exists a function $\epsilon_2(\cdot)$,
defined in a neighborhood of zero, such that
$\hat I(\epsilon_1,\epsilon_2(\epsilon_1))=0$.
We are now in a position to derive the necessary condition (\ref{E-L}).
Consider the real function $\hat J(\epsilon_1,\epsilon_2)=\mathcal{J}(\hat{y})$.
By hypothesis, $\hat J$ has minimum (or maximum) at $(0,0)$ subject to the constraint
$\hat I(0,0)=0$, and we have proved that $\nabla \hat I(0,0)\neq \textbf{0}$.
We can appeal to the Lagrange multiplier rule
(see, \textrm{e.g.}, \cite[p.~77]{Brunt}) to assert
the existence of a number $\lambda$ such that
$\nabla(\hat J(0,0)-\lambda \hat I(0,0))=\textbf{0}$.
Repeating the calculations as before,
$$
\left.\frac{\partial \hat J}{\partial \epsilon_1} \right|_{(0,0)}
=\int_a^b \left[ \frac{\partial L}{\partial y}+{_x\mathcal{D}_b^\alpha}
\frac{\partial L}{\partial u} + {_a\mathcal{D}_x^\beta}
\frac{\partial L}{\partial v}   \right]\eta_1(x)dx
$$
and
$$
\left.\frac{\partial \hat I}{\partial \epsilon_1} \right|_{(0,0)}
=\int_a^b \left[ \frac{\partial g}{\partial y}+{_x\mathcal{D}_b^\alpha}
\frac{\partial g}{\partial u} + {_a\mathcal{D}_x^\beta}
\frac{\partial g}{\partial v}   \right]\eta_1(x)dx \, .
$$
Therefore, one has
\begin{equation}
\label{eq:fae1}
\int_a^b \left[ \frac{\partial L}{\partial y}+{_x\mathcal{D}_b^\alpha}
\frac{\partial L}{\partial u} + {_a\mathcal{D}_x^\beta} \frac{\partial L}{\partial v}
-\lambda \left( \frac{\partial g}{\partial y}+{_x\mathcal{D}_b^\alpha}
\frac{\partial g}{\partial u} + {_a\mathcal{D}_x^\beta}
\frac{\partial g}{\partial v}\right)\right]\eta_1(x)dx=0.
\end{equation}
Since (\ref{eq:fae1}) holds for any function $\eta_1$, we obtain (\ref{E-L}):
$$
\frac{\partial L}{\partial y}+{_x\mathcal{D}_b^\alpha} \frac{\partial L}{\partial u}
+ {_a\mathcal{D}_x^\beta} \frac{\partial L}{\partial v}
-\lambda \left( \frac{\partial g}{\partial y}+{_x\mathcal{D}_b^\alpha}
\frac{\partial g}{\partial u} + {_a\mathcal{D}_x^\beta}
\frac{\partial g}{\partial v}\right)=0.
$$
\end{proof}

\begin{rem}
Theorem~\ref{IsoPro} holds true in the case when $\alpha$
or $\beta$ are equal to $1$. Indeed, in the proof
we imposed the condition $\eta_2(a)=\eta_2(b)=0$,
and formula (\ref{integracao:partes}) is valid.
\end{rem}

\begin{example}
Let $\alpha$ be a given number in the interval $(0,1)$.
We consider the following fractional isoperimetric problem:
\begin{equation}
\label{eq:ex}
\begin{gathered}
\int_0^1(x^4+({_0\mathcal{D}_x^\alpha}y)^2)dx \longrightarrow \min\\
 \int_0^1 x^2{_0\mathcal{D}_x^\alpha}y \,dx=\frac{1}{5}\\
y(0) = 0 \, , \quad y(1) = \frac{2}{2\alpha+3\alpha^2+\alpha^3} \, .
\end{gathered}
\end{equation}
The augmented Lagrangian is
$$
F(x,y,{_0\mathcal{D}_x^\alpha}y,{_x\mathcal{D}_1^\beta}y)
=x^4+({_0\mathcal{D}_x^\alpha}y )^2-\lambda \, x^2 {_0\mathcal{D}_x^\alpha}y
$$
and it is a simple exercise to see that
\begin{equation}
\label{sol:ex:fip}
y(x) = \frac{1}{\Gamma(\alpha)}\int_0^x\frac{t^2}{(x-t)^{1-\alpha}}dt
= \displaystyle\frac{1}{\Gamma(\alpha)} \,
\frac{2x^{\alpha+2}}{2\alpha+3\alpha^2+\alpha^3}
\end{equation}
(i) is not an extremal for the isoperimetric functional,
(ii) satisfy ${_0\mathcal{D}_x^\alpha}y=x^2$,
(iii) (\ref{E-L}) holds for $\lambda=2$, \textrm{i.e.},
${_x\mathcal{D}_1^\alpha}(2 \,
{_0\mathcal{D}_x^\alpha}y-2x^2)=0$.
We remark that for $\alpha= 1$ \eqref{sol:ex:fip} gives
$y(x)=x^3/3$, which coincides with the solution of the
associated classical variational problem (Fig.~\ref{fig:1}).
\begin{figure}[h]
\begin{center}
\pspicture(7,7)(0.5,0.5)
\rput(3,3){\includegraphics[scale=0.4]{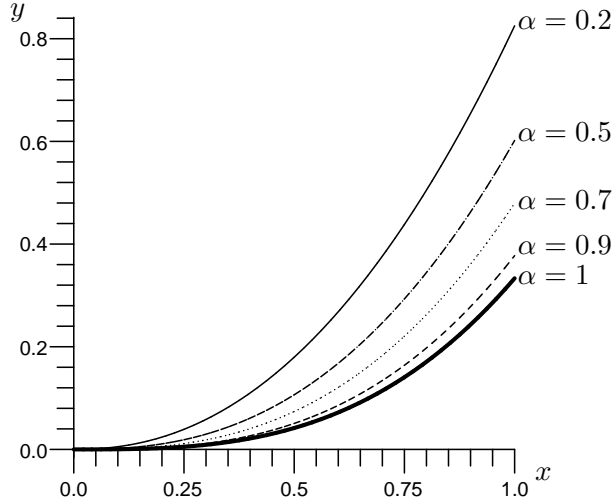}}
\uput{-1}[135](8.3,5.5){$\alpha=0.2$}
\uput{-1}[135](8.3,4){$\alpha=0.5$}
\uput{-1}[135](8.3,3.1){$\alpha=0.7$}
\uput{-1}[135](8.3,2.5){$\alpha=0.9$}
\uput{-1}[135](8,2.1){$\alpha=1$}
\uput{-1}[135](7.5,-0.5){$x$}
\uput{-1}[135](0.5,5.6){$y$}
\endpspicture
\end{center}
\caption{The fractional solution converges to the classical one
as $\alpha\to1$.}
\label{fig:1}
\end{figure}
Indeed, for $\alpha \rightarrow 1$ our fractional problem
(\ref{eq:ex}) tends to the classical isoperimetric problem
of minimizing the functional $\int_0^1(x^4+(y')^2)dx$ subject
to the isoperimetric constraint $\int_0^1 x^2y' \,dx=\frac{1}{5}$
and the boundary conditions $y(0) = 0$ and $y(1)=1/3$.
Then, $F=x^4+(y')^2-\lambda x^2y'$ and the classical
Euler-Lagrange equation is
\begin{equation}
\label{eq:celeq:ex}
\frac{\partial F}{\partial y}
-\frac{d}{dx}\left(\frac{\partial F}{\partial y'}\right)=0
\Leftrightarrow -2y''+2\lambda x=0.
\end{equation}
The solution of (\ref{eq:celeq:ex}) subject to
$y(0) = 0$, $y(1)=1/3$, and $\int_0^1 x^2y' \,dx=\frac{1}{5}$
is $\lambda=2$ and $y=x^3/3$.
\end{example}

Introducing a multiplier $\lambda_0$ associated with
the cost functional (\ref{P0}), we can easily include
in Theorem~\ref{IsoPro} the situation when the solution
of the fractional isoperimetric problem defined by
(\ref{P0})-(\ref{bouncond})-(\ref{isocons}) is an extremal
for the fractional isoperimetric functional.
This is done in Theorem~\ref{IsoPro2}.

\begin{thm}
\label{IsoPro2}
If $y$ is a local minimizer or a local maximizer of (\ref{P0})
subject to the boundary conditions (\ref{bouncond})
and the isoperimetric constraint (\ref{isocons}),
then there exist two constants $\lambda_0$ and $\lambda$,
not both zero, such that
\begin{equation}
\label{E-L2}
\frac{\partial K}{\partial y}+{_x\mathcal{D}_b^\alpha}
\frac{\partial K}{\partial u}+{_a\mathcal{D}_x^\beta}\frac{\partial K}{\partial v}=0
\end{equation}
with $K=\lambda_0 L-\lambda g$.
\end{thm}

\begin{proof}
Using the same notation as in the proof of Theorem~\ref{IsoPro},
we have that $(0,0)$ is an extremal of $\hat J$
subject to the constraint $\hat I=0$. Then, by the abnormal Lagrange multiplier rule
(see, \textrm{e.g.}, \cite[p.~82]{Brunt}) there exist two reals $\lambda_0$ and $\lambda$,
not both zero, such that $\nabla(\lambda_0 \hat J(0,0)-\lambda \hat I(0,0))=\textbf{0}$.
Therefore,
$$
\lambda_0 \displaystyle\left.\frac{\partial \hat J}{\partial \epsilon_1}\right|_{(0,0)}
-\lambda \displaystyle\left.\frac{\partial \hat I}{\partial \epsilon_1}\right|_{(0,0)}=0.
$$
Applying the same reasoning as in the proof of Theorem~\ref{IsoPro},
we end up with (\ref{E-L2}).
\end{proof}


\subsection{An extension}
\label{sec4}

In \cite{Atanackovic} a fractional functional
\begin{equation}
\label{eq:funct:Atanackovic}
\mathcal{L}(y)=\int_A^B L(x,y,\,{_a\mathcal{D}_x^\alpha} y)dx
\end{equation}
is considered with $[A,B]\subset[a,b]$, \textrm{i.e.},
with the lower bound of the integral not coinciding with the lower bound
of the fractional derivative. The main result of \cite{Atanackovic}
is a new Euler-Lagrange equation for the functional \eqref{eq:funct:Atanackovic}.
We now extend the techniques of \cite{Atanackovic}
to prove an Euler-Lagrange equation for functionals containing both left
and right Riemann-Liouville fractional derivatives, \textrm{i.e.},
for fractional functionals of the form
\begin{equation}
\label{Funct}
\mathcal{J}(y)=\int_A^B
L(x,y,\,{_a\mathcal{D}_x^\alpha} y,\, {_x\mathcal{D}_b^\beta}y)dx,
\end{equation}
where the integrand $L$ satisfies the same conditions as before.
Let $y$ be a local extremum of $\mathcal{J}$ such that $y(a)=y_a$ and $y(b)=y_b$,
and let $\hat y=y+\epsilon\eta$ with $\eta(a)=\eta(b)=0$. Consider the function
$\hat J (\epsilon)=\mathcal{J}(y+\epsilon\eta)$.
Since $\hat J (\epsilon)$ has a local extremum at $\epsilon=0$, then
\begin{equation*}
\begin{split}
0&= \int_A^B \left[\frac{\partial L}{\partial y} \cdot \eta
+\frac{\partial L}{\partial u}\cdot {_a\mathcal{D}_x^\alpha} \eta
+ \frac{\partial L}{\partial v}\cdot {_x\mathcal{D}_b^\beta}\eta\right]dx\\
&=  \int_A^B \frac{\partial L}{\partial y} \cdot \eta dx
+\left[\int_a^B \frac{\partial L}{\partial u}\cdot {_a\mathcal{D}_x^\alpha} \eta dx
- \int_a^A \frac{\partial L}{\partial u}\cdot {_a\mathcal{D}_x^\alpha} \eta dx \right]\\
&\quad +  \left[ \int_A^b \frac{\partial L}{\partial v}
\cdot {_x\mathcal{D}_b^\beta}\eta dx - \int_B^b \frac{\partial L}{\partial v}
\cdot {_x\mathcal{D}_b^\beta}\eta dx  \right]\\
&= \int_A^B \frac{\partial L}{\partial y} \cdot \eta dx
+\left[\int_a^B \eta \cdot {_x\mathcal{D}_B^\alpha}\frac{\partial L}{\partial u} dx
- \int_a^A \eta \cdot {_x\mathcal{D}_A^\alpha}\frac{\partial L}{\partial u} dx \right]\\
&\quad + \left[ \int_A^b \eta
\cdot {_A\mathcal{D}_x^\beta}\frac{\partial L}{\partial v} dx
- \int_B^b\eta \cdot {_B\mathcal{D}_x^\beta}\frac{\partial L}{\partial v} dx  \right].
\end{split}
\end{equation*}
Continuing in a similar way,
\begin{equation*}
\begin{split}
0&= \int_A^B \frac{\partial L}{\partial y}
\cdot \eta dx +\left[\int_a^A \eta
\cdot {_x\mathcal{D}_B^\alpha}\frac{\partial L}{\partial u} dx
+ \int_A^B \eta \cdot {_x\mathcal{D}_B^\alpha}\frac{\partial L}{\partial u} dx
- \int_a^A \eta \cdot {_x\mathcal{D}_A^\alpha}\frac{\partial L}{\partial u} dx \right]\\
&\quad + \left[ \int_A^B \eta
\cdot {_A\mathcal{D}_x^\beta}\frac{\partial L}{\partial v} dx
+\int_B^b \eta \cdot {_A\mathcal{D}_x^\beta}\frac{\partial L}{\partial v} dx
- \int_B^b\eta \cdot {_B\mathcal{D}_x^\beta}\frac{\partial L}{\partial v} dx  \right]\\
&= \int_a^A\left[ {_x\mathcal{D}_B^\alpha}\frac{\partial L}{\partial u}
- {_x\mathcal{D}_A^\alpha}\frac{\partial L}{\partial u} \right] \eta dx
+ \int_A^B \left[ \frac{\partial L}{\partial y}
+{_x\mathcal{D}_B^\alpha}\frac{\partial L}{\partial u}
+ {_A\mathcal{D}_x^\beta}\frac{\partial L}{\partial v}\right]\eta dx\\
&\quad + \int_B^b \left[{_A\mathcal{D}_x^\beta}\frac{\partial L}{\partial v}
- {_B\mathcal{D}_x^\beta}\frac{\partial L}{\partial v}\right]\eta dx \, .
\end{split}
\end{equation*}
Let $\eta_1:[a,A]\to\mathbb R$ be any function
satisfying $\eta_1(a)=0$, and $\eta$ be given by
$$
\eta(x)=\left\{
\begin{array}{ll}
\eta_1(x) & \mbox{ if } x\in[a,A] \, ,\\
0 & \mbox{ elsewhere. }\\
\end{array}\right.
$$
Therefore,
$$
0= \int_a^A\left[ {_x\mathcal{D}_B^\alpha}\frac{\partial L}{\partial u}
- {_x\mathcal{D}_A^\alpha}\frac{\partial L}{\partial u} \right] \eta_1 dx.
$$
By the arbitrariness of $\eta_1$ and the fundamental lemma of calculus of variations,
$$
{_x\mathcal{D}_B^\alpha}\frac{\partial L}{\partial u}
- {_x\mathcal{D}_A^\alpha}\frac{\partial L}{\partial u}=0 \mbox{ for all } x\in[a,A].
$$
Analogously, we have
$$
\frac{\partial L}{\partial y}+{_x\mathcal{D}_B^\alpha}\frac{\partial L}{\partial u}
+ {_A\mathcal{D}_x^\beta}\frac{\partial L}{\partial v}=0 \mbox{ for all } x\in[A,B] \, ,
$$
and
$$
{_A\mathcal{D}_x^\beta}\frac{\partial L}{\partial v}
- {_B\mathcal{D}_x^\beta}\frac{\partial L}{\partial v}
=0 \mbox{ for all } x\in[B,b].
$$
We have just proved the following.

\begin{thm}
\label{ext:ELFE}
Let $y$ be a local extremizer of (\ref{Funct}).
Then, $y$ satisfies the following equations:
$$
\left\{\begin{array}{ll}
\displaystyle \frac{\partial L}{\partial y}
+{_x\mathcal{D}_B^\alpha} \frac{\partial L}{\partial u}
+{_A\mathcal{D}_x^\beta}\frac{\partial L}{\partial v}=0
& \mbox{ for all } x \in [A,B]\, ,\\[0.3cm]
\displaystyle {_x\mathcal{D}_B^\alpha} \frac{\partial L}{\partial u}
-{_x\mathcal{D}_A^\alpha}\frac{\partial L}{\partial u}=0
& \mbox{ for all } x \in [a,A]\, , \\[0.3cm]
\displaystyle {_A\mathcal{D}_x^\beta} \frac{\partial L}{\partial v}
-{_B\mathcal{D}_x^\beta}\frac{\partial L}{\partial v}=0 & \mbox{ for all } x \in [B,b].
\end{array}\right.
$$
\end{thm}

\begin{rem}
Theorem~\ref{ext:ELFE} simplifies to the result
proved in \cite{Atanackovic}
in the case the Lagrangian $L$ in (\ref{Funct})
does not depend on the right Riemann-Liouville
fractional derivative ${_x\mathcal{D}_b^\beta}y$.
\end{rem}

We will study now the fractional isoperimetric problem for functionals
of type (\ref{Funct}) subject to an integral constraint
\begin{equation}
\label{Constraint}
\mathcal{I}(y)=\int_A^B g(x,y,\, {_a\mathcal{D}_x^\alpha} y,\,
{_x\mathcal{D}_b^\beta} y)dx=l \, .
\end{equation}

\begin{defn}
We say that $y$ is an \emph{extremal} for functional
$\mathcal{I}$ given in (\ref{Constraint}) if
$$
\frac{\partial g}{\partial y}
+{_x\mathcal{D}_B^\alpha} \frac{\partial g}{\partial u}
+{_A\mathcal{D}_x^\beta}\frac{\partial g}{\partial v}=0
\quad \mbox{ for all } x\in[A,B].
$$
\end{defn}

\begin{thm}
\label{thm:gt:n}
Let $y$ give a local minimum or a local maximum to the fractional
functional (\ref{Funct}) subject to the constraint (\ref{Constraint}).
If $y$ is not an extremal for $\mathcal{I}$,
then there exists a constant $\lambda$ such that
\begin{equation}
\label{IsoPro3}
\left\{\begin{array}{ll}
\displaystyle \frac{\partial F}{\partial y}
+{_x\mathcal{D}_B^\alpha} \frac{\partial F}{\partial u}
+{_A\mathcal{D}_x^\beta}\frac{\partial F}{\partial v}=0
& \mbox{ for all } x \in [A,B]\vspace{0.1cm}\\
\displaystyle {_x\mathcal{D}_B^\alpha} \frac{\partial F}{\partial u}
-{_x\mathcal{D}_A^\alpha}\frac{\partial F}{\partial u}=0
& \mbox{ for all } x \in [a,A]\vspace{0.2cm}\\
\displaystyle {_A\mathcal{D}_x^\beta} \frac{\partial F}{\partial v}
-{_B\mathcal{D}_x^\beta}\frac{\partial F}{\partial v}=0
& \mbox{ for all } x \in [B,b]
\end{array}\right.
\end{equation}
with $F=L-\lambda g$.
\end{thm}

\begin{proof}
Consider a variation
$(\epsilon_1,\epsilon_2)\mapsto \hat y =y+\epsilon_1 \eta_1+\epsilon_2\eta_2$
where $\eta_1(a)=\eta_1(b)=\eta_2(a)=\eta_2(b)=0$. Let
$$
\hat I(\epsilon_1,\epsilon_2)=\int_A^B g(x,\hat y,\,
{_a\mathcal{D}_x^\alpha}\hat y,\,{_x\mathcal{D}_b^\beta}\hat y)dx-l.
$$
Then, $\hat I(0,0)=0$ and
$$
\begin{array}{ll}
\displaystyle \left.\frac{\partial \hat I}{\partial \epsilon_2} \right|_{(0,0)}
&=\displaystyle\int_A^B\left[ \frac{\partial g}{\partial y}\eta_2
+\frac{\partial g}{\partial u} {_a\mathcal{D}_x^\alpha} \eta_2
+ \frac{\partial g}{\partial v} {_x\mathcal{D}_b^\beta} \eta_2 \right]dx\\
&= \displaystyle \int_a^A\left[ {_x\mathcal{D}_B^\alpha}\frac{\partial g}{\partial u}
- {_x\mathcal{D}_A^\alpha}\frac{\partial g}{\partial u} \right] \eta_2 dx
+ \int_A^B \left[ \frac{\partial g}{\partial y}
+{_x\mathcal{D}_B^\alpha}\frac{\partial g}{\partial u}
+ {_A\mathcal{D}_x^\beta}\frac{\partial g}{\partial v}\right]\eta_2 dx\\
&\quad +\displaystyle \int_B^b \left[{_A\mathcal{D}_x^\beta}\frac{\partial g}{\partial v}
- {_B\mathcal{D}_x^\beta}\frac{\partial g}{\partial v}\right]\eta_2 dx.
\end{array}
$$
Since $y$ is not an extremal for $\mathcal{I}$, there exists a function $\eta_2$ such that
$\left.\frac{\partial \hat I}{\partial \epsilon_2} \right|_{(0,0)}\neq 0$.
By the implicit function theorem, there exists a subset of curves
$\{ y+\epsilon_1 \eta_1 + \epsilon_2\eta_2 \, | \, (\epsilon_1,\epsilon_2)\in \mathbb R^2 \}$
admissible for the fractional isoperimetric problem.
Let $\hat J(\epsilon_1,\epsilon_2)=\mathcal{J}(\hat{y})$.
Then, there exists a real $\lambda$ such that
$\nabla(\hat J(0,0)-\lambda \hat I(0,0))=\textbf{0}$. Because
$$
\begin{array}{ll}
\displaystyle\left.\frac{\partial \hat J}{\partial \epsilon_1} \right|_{(0,0)}
&=\displaystyle \int_a^A\left[ {_x\mathcal{D}_B^\alpha}\frac{\partial L}{\partial u}
- {_x\mathcal{D}_A^\alpha}\frac{\partial L}{\partial u} \right] \eta_1 dx
+ \int_A^B \left[ \frac{\partial L}{\partial y}+{_x\mathcal{D}_B^\alpha}\frac{\partial L}{\partial u}
+ {_A\mathcal{D}_x^\beta}\frac{\partial L}{\partial v}\right]\eta_1 dx\\
&\quad +\displaystyle\int_B^b \left[{_A\mathcal{D}_x^\beta}\frac{\partial L}{\partial v}
- {_B\mathcal{D}_x^\beta}\frac{\partial L}{\partial v}\right]\eta_1 dx,
\end{array}
$$
$$
\begin{array}{ll}
\displaystyle\left.\frac{\partial \hat I}{\partial \epsilon_1}
\right|_{(0,0)}&=\displaystyle \int_a^A\left[ {_x\mathcal{D}_B^\alpha}\frac{\partial g}{\partial u}
- {_x\mathcal{D}_A^\alpha}\frac{\partial g}{\partial u} \right] \eta_1 dx
+ \int_A^B \left[ \frac{\partial g}{\partial y}+{_x\mathcal{D}_B^\alpha}\frac{\partial g}{\partial u}
+ {_A\mathcal{D}_x^\beta}\frac{\partial g}{\partial v}\right]\eta_1 dx\\
&\quad +\displaystyle\int_B^b \left[{_A\mathcal{D}_x^\beta}\frac{\partial g}{\partial v}
- {_B\mathcal{D}_x^\beta}\frac{\partial g}{\partial v}\right]\eta_1 dx,
\end{array}
$$
$$
\left.\frac{\partial \hat J}{\partial \epsilon_1} \right|_{(0,0)}
-\lambda \left.\frac{\partial \hat I}{\partial \epsilon_1} \right|_{(0,0)}=0 ,
$$
and $\eta_1$ is an arbitrary function, it follows  (\ref{IsoPro3}).
\end{proof}

Similarly as before, we can include in Theorem~\ref{thm:gt:n} the situation
when the solution $y$ is an extremal for $\mathcal{I}$ (abnormal extremizer).
For that we introduce a new multiplier $\lambda_0$ that will be zero
when the solution $y$ is an extremal for $\mathcal{I}$ and one otherwise.

\begin{thm}
If $y$ is a local minimizer or a local maximizer of (\ref{Funct})
subject to the isoperimetric constraint (\ref{Constraint}),
then there exist two constants $\lambda_0$ and $\lambda$,
not both zero, such that
$$
\left\{\begin{array}{ll}
\displaystyle\frac{\partial K}{\partial y}+{_x\mathcal{D}_B^\alpha} \frac{\partial K}{\partial u}
+{_A\mathcal{D}_x^\beta}\frac{\partial K}{\partial v}=0 & \mbox{ for all } x \in [A,B]\vspace{0.1cm}\\
\displaystyle{_x\mathcal{D}_B^\alpha} \frac{\partial K}{\partial u}
-{_x\mathcal{D}_A^\alpha}\frac{\partial K}{\partial u}=0 & \mbox{ for all } x \in [a,A]\vspace{0.2cm}\\
\displaystyle{_A\mathcal{D}_x^\beta} \frac{\partial K}{\partial v}
-{_B\mathcal{D}_x^\beta}\frac{\partial K}{\partial v}=0 & \mbox{ for all } x \in [B,b]\\
\end{array}\right.
$$
with $K=\lambda_0 L-\lambda g$.
\end{thm}


\subsection{Dependence on a parameter}
\label{sec5}

Consider the following fractional problem of the calculus
of variations: to extremize the functional
$$
\Psi(y)=\int_0^1\left[\frac{x^\alpha}{\Gamma(\alpha+1)}({_0D^\alpha_x}y)^2
-2\overline y \, {_0D^\alpha_x}y  \right]^2dx
$$
when subject to the boundary conditions
$$
y(0)=0, \quad y(1)=1.
$$
Here, $\overline y := x^\alpha, x\in[0,1]$.
The fractional Euler-Lagrange associated to this problem is
\begin{equation}
\label{FELequation}
{_xD^\alpha_1}\left(2 \left[\frac{x^\alpha}{\Gamma(\alpha+1)}({_0D^\alpha_x}y)^2
-2\overline y \, {_0D^\alpha_x} y \right]
\cdot \left[\frac{2 x^\alpha}{\Gamma(\alpha+1)}{_0D^\alpha_x}y
-2\overline y \right]\right)=0.
\end{equation}
Replacing $y$ by $\overline y$, and since ${_0D^\alpha_x}\overline y=\Gamma(\alpha+1)$,
we conclude that $\overline y$ is a solution of (\ref{FELequation}).

Consider now the following problem: what is the order of the derivative $\alpha$,
such that $\Psi(\overline y)$ attains a maximum or a minimum? In other words,
find the extremizers for $\psi(\alpha)=\Psi(\overline y)$.
Direct computations show that
$$
\psi(\alpha)=\int_0^1\left[x^\alpha \Gamma(\alpha+1)\right]^2dx.
$$
\begin{figure}[ht]
\centering
\includegraphics[width=8cm]{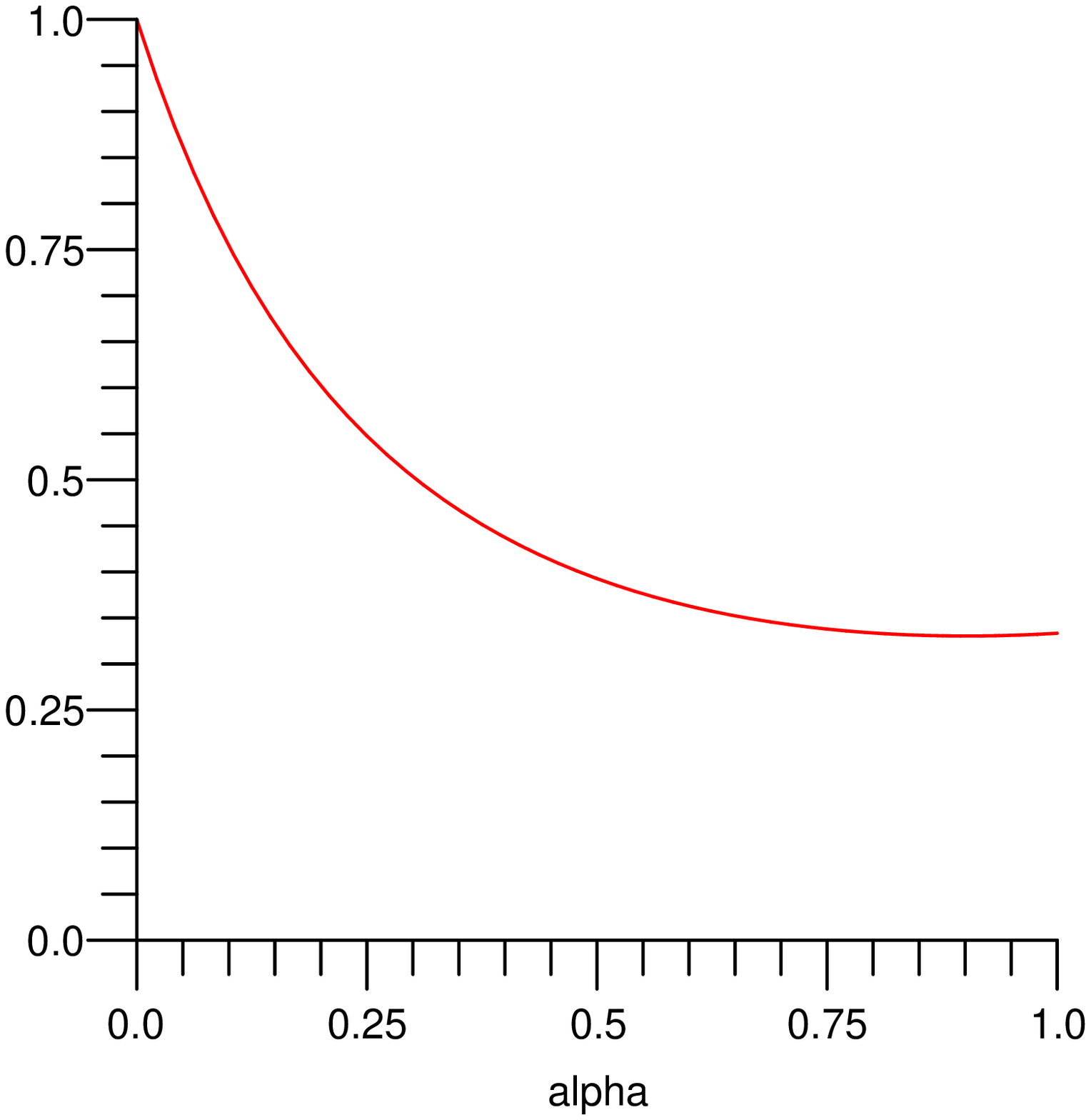}\\
\vspace{-1cm}\caption{Graph of $\Psi(\overline y)$ for
$\alpha \in [0,1]$}
\end{figure}
Evaluating its derivative,
$$
\begin{array}{ll}
\psi'(\alpha)
& = \displaystyle \int_0^1\frac{d}{d\alpha}\left[x^\alpha \Gamma(\alpha+1)\right]^2dx\\
&\\
& = \displaystyle\int_0^1 2 x^\alpha \Gamma(\alpha+1) \left[x^\alpha\ln x \,
\Gamma(\alpha+1)+x^\alpha \int_0^\infty t^\alpha \ln t \, e^{-t}dt\right] dx.
\end{array}
$$
We have that $\alpha\approx0.901$ is a solution of the equation $\psi'(\alpha)=0$,
and such value is precisely where $\Psi(\overline y)$ attains a minimum.

More generally, consider the functional
\begin{equation}
\label{eq:prb:par}
\Phi(y,\alpha)=\int_a^b L(x,y(x),\,{_aD_x^\alpha} y(x))dx.
\end{equation}
Functional (\ref{eq:prb:par}) contains the left Riemann-Liouville derivative only,
but we can consider functionals containing right Riemann-Liouville derivatives or both
in a similar way.
Let $h$ be a curve such that $h(a)=h(b)=0$, $\delta$ be a real number,
and $(y,\alpha)$ be an extremal for $\Phi$. Then,
\begin{multline*}
\Phi(y+h,\alpha+\delta)-\Phi(y,\alpha)\\
=\int_a^b \frac{\partial L}{\partial y} \cdot h + \frac{\partial L}{\partial u}
\cdot {_aD_x^{\alpha+\delta}} h+ \frac{\partial L}{\partial u}
\cdot ({_aD_x^{\alpha+\delta}} y-{_aD_x^{\alpha}} y) dx+O|(h,\delta)|^2 \, .
\end{multline*}

For $\delta=0$, using the fractional integration by parts formula
and the fundamental lemma of the calculus of variations,
we obtain the known fractional Euler-Lagrange equation:
$$
\frac{\partial L}{\partial y}(x,y(x),\,{_aD_x^\alpha} y(x))
+ {_xD_b^{\alpha}} \frac{\partial L}{\partial u}(x,y(x),\,{_aD_x^\alpha} y(x))=0.
$$
For $h=0$, we obtain the relation
$$
\int_a^b \frac{\partial L}{\partial u}(x,y(x),\,{_aD_x^\alpha} y(x))\phi'(\alpha)dx=0,
$$
where $\phi(\alpha)={_aD_x^\alpha} y(x)$. In summary, we have:

\begin{thm}
If $(y,\alpha)$ is an extremal of $\Phi$ given by \eqref{eq:prb:par},
satisfying the boundary conditions $y(a)=0$ and $y(b)=y_b$,
then $y$ satisfies the system
\begin{equation}
\label{system}
\left\{\begin{array}{l}
\displaystyle \frac{\partial L}{\partial y}(x,y(x),\,{_aD_x^\alpha} y(x))
+ {_xD_b^{\alpha}} \frac{\partial L}{\partial u}(x,y(x),\,{_aD_x^\alpha} y(x))=0\\
\displaystyle \int_a^b \frac{\partial L}{\partial u}(x,y(x),\,{_aD_x^\alpha} y(x))\phi'(\alpha)dx=0\\
\end{array}\right.\end{equation}
where $\phi(\alpha)={_aD_x^\alpha} y(x)$.
\end{thm}

In the previous example the solution obtained satisfies system (\ref{system}) since
$$\frac{\partial L}{\partial u}(x,\overline y(x),\,{_aD_x^\alpha} \overline y(x)))=0.$$


\section*{Acknowledgements}

Work partially supported by the {\it Centre for Research on
Optimization and Control} (CEOC) from the ``Funda\c{c}\~{a}o para a
Ci\^{e}ncia e a Tecnologia'' (FCT), cofinanced by the European
Community Fund FEDER/POCI 2010. The second author was also
supported by FCT through the PhD fellowship SFRH/BD/39816/2007.



\end{document}